\documentclass[a4paper,12pt]{article}
\usepackage{amsthm}
\usepackage{amsfonts}
\usepackage{amsmath}
\usepackage{amssymb}
\usepackage{hyperref}
\usepackage{datetime}
\usepackage{diagrams}
\usepackage{commath}
\usepackage{float}
\restylefloat{table}

\numberwithin{equation}{section}

\newtheorem{theorem}[equation]{Theorem}
\theoremstyle{definition}

\theoremstyle{remark}

\theoremstyle{plain}
\newtheorem{lemma}[equation]{Lemma}

\newtheorem{proposition}[equation]{Proposition}

\newcommand{\C}{\mathbb{C}}

\newcommand{\diag}{\text{diag}}
\newcommand{\bq}{/\!\!/}

\begin{document}

\title{The classification of $SU(2)^2$ biquotients of rank 3 Lie groups}

\author{Jason DeVito and Robert L. DeYeso III}

\date{}

\maketitle

\begin{abstract}

We classify all compact simply connected biquotients of the form $G\bq SU(2)^2$ for $G =SU(4), SO(7), Spin(7)$, or $G = \mathbf{G}_2\times SU(2)$.  In particular, we show there are precisely $2$ inhomogeneous reduced biquotients in the first and last case, and $10$ in the middle cases.

\end{abstract}

\section{Introduction}\label{sec:intro}

Geometrically, a biquotient is any manifold diffeomorphic to the quotient of a Riemannian homogeneous space $G/H$ by a free isometric action of a subgroup $K\subseteq \operatorname{Iso}(G/H)$; the resulting quotient is denoted $K\backslash G/H$.  Biquotients also have a purely Lie theoretic description:  if $ U$ is a compact Lie group, then any $f=(f_1,f_2):U\rightarrow G\times G$ defines an action of $U$ on $G$ via $u\ast g= f_1(u) \, g\, f_2(u)^{-1}$.  If this action is free, the resulting quotient $G\bq U$ is called a biquotient.

In general, biquotients are not even homotopy equivalent to homogeneous spaces.  Nevertheless, one may often compute their geometry and topology, making them a prime source of examples.

If $G$ is a compact Lie group equipped with a bi-invariant metric, then $U$ acts isometrically, and thus $G\bq U$ inherits a metric.  By O'Neill's formula \cite{On1}, the resulting metric on $G\bq U$ has non-negative sectional curvature.  In addition, until the recent example of Grove, Verdiani and Ziller \cite{GVZ}, and independently Dearicott \cite{De}, all known examples of positively curved manifolds were constructed as biquotients \cite{AW, Ber,Es1, Baz1,Wa}.  Further, almost all known examples of quasi-positively curved and almost positively curved manifolds are constructed as biquotients.  See \cite{DDRW,D1,Ta1,KT,EK,Ke1,Ke2,PW2,W,Wi} for these examples.

Biquotients were first discovered by Gromoll and Meyer \cite{GrMe1} when the exhibited an exotic sphere as a biquotient.

In his Habilitation, Eschenburg \cite{Es2} classified all biquotients $G\bq U$ with $G$ compact simply connected of rank $2$.  We partially extend his classification when $G$ has rank $3$.  Using the well known classification of simple Lie groups together with the low dimensional exceptional isomorphisms, one easily sees that, up to cover, the complete list of rank $3$ semi-simple Lie groups is $SU(2)^3$, $\mathbf{G}_2\times SU(2)$, $SU(3)\times SU(2)$, $Sp(2)\times SU(2)$, $Sp(3)$, $SU(4)$, and $SO(7)$.

In his thesis, the first author \cite{D1} classified all compact simply connected biquotients of dimension at most $7$; these include all examples of the form $G_1\times SU(2)/SU(2)^2$ where $G_1$ has rank $2$, except $G_1 =\mathbf{G}_2$.  In addition, the authors, together with Ruddy and Wesner \cite{DDRW}, have classified all biquotients of the form $Sp(3) \bq SU(2)^2$.

\begin{theorem}\label{main}  For $U = SU(2)^2$, $G = SU(4), SO(7)$, or $\mathbf{G}_2\times SU(2)$,  there are, respectively, precisely $2$, $10$ and $2$ reduced inhomogeneous biquotients of the form $G\bq U$.  Table \ref{table:su2class} lists them all.  Further, a biquotient action of $U$ on $Spin(7)$ is effectively free iff it is the lift of an effectively free biquotient action of $U$ on $SO(7)$.

\end{theorem}

\begin{table}[h]

\begin{center}

\begin{tabular}{|c|c|c|}

\hline

 G & Left image & Right image\\
 
 \hline
 \hline
 
 $SU(4)$ & $\diag(A,A)$ & $\diag(B,I_2)$ \\
 
 \hline
 
 $SU(4)$ & $\diag(A,A)$ & $\diag(\pi(B), 1)$ \\
 
 \hline
 
 \hline
 
 \hline
 
 $SO(7)$ & $\diag (\pi(A), I_4)$ & $\diag(\pi(B), i(A))$ \\
 
 \hline
 
 $SO(7)$ & $\diag(\pi(A), I_4)$ & $\diag(i(B), I_3)$ \\
 
 \hline
 
 $SO(7)$ & $\diag(\pi(A), I_4)$ & $\diag(\pi(B), \pi(B), 1)$ \\
 
 \hline
 
 $SO(7)$ & $\diag(\pi(A), I_4)$ & $ \diag(\pi(B), i(B))$ \\
 
 \hline
 
 $SO(7)$ & $\diag(\pi(A), I_4)$ & $ \diag(\pi(A,B), \pi(A))$ \\
 
 \hline
 
 $SO(7)$ & $\diag(\pi(A), I_4)$ & $\diag(\pi(A,B), \pi(B))$ \\
 
 \hline
 
 $SO(7)$ & $\diag(\pi(A),I_4)$  & $ B_{max} $ \\
 
 \hline
 
 $SO(7)$ & $\diag(\pi(A), i(B))$ & $ \diag(\pi(A), \pi(A), 1)$ \\
 
 \hline
 
 $SO(7)$ & $\diag(\pi(A), \pi(B), 1)$ & $\diag(\pi(A), i(A))$\\
 
 \hline
 
 $SO(7)$ & $\diag(A_{Ber},1 , 1)$ & $ \diag(\pi(B), i(B))$ \\
 
 \hline 
 
 \hline
 
 \hline
 
 $\mathbf{G}_2\times SU(2)$ & $(\pi(A,B), A)$ & $(I, A)$\\
 
 \hline
 
 $\mathbf{G}_2\times SU(2)$ & $(\pi(A,B),B)$ & $(I,B)$ \\

 \hline

\end{tabular}

\caption{Biquotients of the form $G\bq SU(2)^2$ with $G$ rank $3$ and simple or $G = \mathbf{G}_2\times SU(2)$}
\label{table:su2class}

\end{center}

\end{table}

In Table \ref{table:su2class}, $I_k$ denotes the $k\times k$ identity matrix, and $\pi$, depending on the number of arguments, denotes either of the canonical double covers $SU(2)\rightarrow SO(3)$  and  $SU(2)^2\rightarrow SO(4)$.  The notation $i(A)$ refers to the standard inclusion $i:SU(2)\rightarrow SO(4)$ obtained by identifying $\mathbb{C}^2$ with $\mathbb{R}^4$.  The notation $B_{max}$ refers to the unique maximal $SO(3)\subseteq SO(7)$, and the notation $A_{Ber}$ refers to the maximal $SO(3)$ in $SO(5)$, whose quotient $B^7 = SO(5)/SO(3)$ is the positively curved Berger space \cite{Ber}.  The term \textit{reduced} refers to the condition that $U$ not act transitively on any factor of $G$, and only applies in the case of $\mathbf{G}_2\times SU(2)$.  Non-reduced biquotients of the form $\mathbf{G}_2\times SU(2)\bq SU(2)^2$ are all diffeomorphic to biquotients of the form $\mathbf{G}_2\bq SU(2)$, and are classified in \cite{KZ}.

It is not a priori clear, but follows from this classification that all biquotients of the form $SO(7)\bq SU(2)^2$ are simply connected.

The lift of an effectively free action on a connected smooth manifold to a connected cover is effectively free, but in general, the cover admits effectively free actions which do not induce effectively free actions on the base.  Our main tool for understanding homomorphisms from $SU(2)^2$ into $Spin(7)$ is utilizing a concrete description, using the octonions and Clifford algebras, of the spin representation $Spin(7)\rightarrow SO(8)$.

The rest of the paper is organized as follows.  Section 2 is devoted to background material on biquotient actions, representation theory, and the octonions.   In Section 3, we prove Theorem \ref{main} in regard to $G=SO(7)$ and $Spin(7)$.  In Section 4, we prove Theorem \ref{main} when $G = SU(4)$ or $\mathbf{G}_2\times SU(2)$.

We would like to thank the anonymous referee for suggesting an alternate approach which significantly simplified Section \ref{oct}.

\section{Background:  Biquotients and Representation Theory}\label{sec:background}

In this section, we cover the necessary background for proving Theorem \ref{main}.

\subsection{Biquotients}

As mentioned in the introduction, given compact Lie groups $U$ and $G$ with $f=(f_1,f_2):U\rightarrow G\times G$ a homomorphism, $U$ naturally acts on $G$ via $u \ast g = f_1(u)\, g\, f_2(u)^{-1}$.

A simple criterion to determine when a biquotient action is effectively free is given by the following proposition.

\begin{proposition}\label{free}  A biquotient action of $U$ on $G$ is effectively free if and only if for any $(u_1,u_2)\in f(U)$, if $u_1$ is conjugate to $u_2$ in $G$, then $u_1 = u_2\in Z(G)$.

\end{proposition}

Since every element of a Lie group $U$ is conjugate to an element in its maximal torus $T_U$, it follows that a biquotient action of $U$ on $G$ is effectively free if and only if the action is effectively free when restricted to $T_U$.

To begin classifying biquotient actions, we note that, as mentioned in \cite{DeV1}, replacing $(f_1, f_2)$ in any one of the following three ways will give an equivalent action:  $(\phi\circ  f_1, \phi \circ f_2)$ for $\phi$ an automorphism of $G$, $(f_1 \circ \psi, f_2 \circ \psi)$ for $\psi$ an automorphism of $U$, or $(C_{g_1} \circ f_1, C_{g_2}\circ f_2)$ where $C_{g_i}$ denotes conjugation by $g_i$.

Hence, we may classify all biquotient actions of $U$ on $G$ by classifying the conjugacy classes of images of homomorphisms from $U$ into $G\times G$ and then checking each of these to see if the induced action is effectively free.  Combining this with Proposition \ref{free}, it follows that if $f(T_U)\subseteq T_{G\times G}$, then the action of $U$ on $G$ is (effectively) free if and only if the induced action of $T_U$ on $T_{G\times G}$ is (effectively) free.

We also note the following proposition, found in \cite{KZ}.

\begin{proposition}Suppose $U$ acts on $G_1$ and $G_2$ and that the action on $G_2$ is transitive with isotropy group $U_0$.  Suppose further that the diagonal action of $U$ on $G_1\times G_2$ is effectively free.  Then the action of $U_0$ on $G_1$ is effectively free and the quotients $(G_1\times G_2)\bq U$ and $G_1\bq U_0$ are canonically diffeomorphic.

\end{proposition}

So we see that, in terms of classifying biquotients, we may assume our biquotients are reduced -- that is, $U$ does not act transitively on any simple factor of $G$.

\

Since the restriction of any effectively free action to a subgroup is effectively free, we have the following simple lemma which we will make repeated use of.

\begin{lemma}\label{restest} Suppose $f=(f_1,f_2):U = SU(2)^2 \rightarrow G^2$ defines an effectively free action of $U$ on $G$.  Then the restriction of $f$ to both factors of $U$, as well as to the diagonal $SU(2)$ in $U$, must define an effectively free action of $SU(2)$ on $G$.

\end{lemma}

We conclude this subsection with a a proposition relating effectively free actions on a connected cover with effectively free actions on a connected base space.

\begin{proposition}\label{easyhalf}

Suppose $\pi:\tilde{M}\rightarrow M$ is an equivariant covering of smooth connected $G$-manifolds.  If the $G$ action on $M$ is effectively free, then it is effectively free on $\tilde{M}$ as well.  Conversely, if the deck group is a subset of $G\subseteq Diff(M)$ and the action on $\tilde{M}$ is effectively free, it is also effectively free on $M$.

\end{proposition}

\begin{proof}

Suppose there is a $g\in G$ and a $p\in \tilde{M}$ with $g\ast p = p$.  Then $g\ast \pi(p) = \pi(g\ast p) =\pi(p)$, so $g$ fixes $\pi(p)$.  Since the action on $M$ is effectively free, $g$  must act trivially on all of $M$.  This implies that, for any $q\in \tilde{M}$, that $\pi(g\ast q) = \pi(q)$, that is, multiplication by $g$ is an element of the deck group of the covering.  Since $g$ fixes $p$, it must thus fix all of $\tilde{M}$.

\

Conversely, suppose the $G$ action on $\tilde{M}$ is effectively free and the deck group is a subgroup of $G$.  If $\pi(p) = g\ast \pi(p) = \pi(g\ast p)$ for some $g\in G$, $p\in \tilde{M}$, then we see that $p = \mu(g\ast p)$ for some $\mu$ in the deck group.  Writing $\mu = g_1\in G$, we have $p = (g_1 g) \ast p$.  Since the action on $\tilde{M}$ is effectively free, we conclude that $g_1 g$ fixes every point of $\tilde{M}$.  Then for any $q\in\tilde{M}$, $\pi(q) = \pi(g_1 g \ast q) = \pi(g\ast q) = g\ast \pi(q)$, that is, $g$ fixes $M$ pointwise.

\end{proof}

We note that the if the hypothesis on the deck group is omitted, then the converse of Proposition \ref{easyhalf} is not true in general for biquotients, though it is for homogeneous actions.  For example, the biquotient action of $SU(2) = Sp(1)$ on $Sp(3)$ given by $p\ast A = \diag(p,1,1) A \diag(1,p,p)^{-1}$ is free, but the induced action on $Sp(3)/\{\pm I\}$ is not effectively free because the element $-1\in Sp(1)$ fixes $[I]\in Sp(3)/\{\pm I\}$ but does not fix $[\diag(R(\theta), 1)]$ unless $\theta$ is an integral multiple of $\pi$.  In this case, the deck group element centralizes the $G$ action.

On the other hand, we will see that the hypothesis on the deck group is not necessary in general:  there is a unique biquotient of the form $Spin(7)\bq SU(2)^2$ for which $SU(2)^2\subseteq Diff(Spin(7))$ does not contain the deck group, but the induced $SU(2)^2$ action on $SO(7)$ is still effectively free.

\subsection{Representation theory}
\label{Reptheory}

Our homomorphisms $f:U\rightarrow G\times G$ will be constructed via representation theory.    The following information can all be found in \cite{FH}.  Recall that a representation of $U$ is a homomorphism $\rho:U\rightarrow Gl(V)$ for some complex vector space $V$.  It is well known that if $U$ is a compact semi-simple Lie group, then $\rho(U)$ is conjugate to a subgroup of $SU(V)$ and that $\rho$ is completely reducible -- every such $\rho$ is a direct sum of irreducible representations.  The representation $\rho$ is called orthogonal if the image is conjugate to a subgroup of the standard $SO(n)\subseteq SU(n)$.  If, $V = \C^{2n}$, and the image of $\rho$ is conjugate to a subgroup of the standard $Sp(n)\subseteq SU(2n)$, $\rho$ is called symplectic.  If $\rho$ is neither orthogonal nor symplectic, it is called complex.

We note that an irreducible representation is complex iff it is not isomorphic to its conjugate representation.  Recall the following well known proposition.

\begin{proposition}\label{symp} A representation $\rho$ is orthogonal (symplectic) if and only if $\rho \cong \bigoplus_i (\psi_i\oplus \overline{\psi}_i)\oplus\ \bigoplus_j \phi_j,$ where each $\phi_j$ is orthogonal (symplectic) and $\overline{\psi}_i$ denotes the conjugate representation of $\psi_i$.

\end{proposition}

Since we are interested in the case $U = SU(2)^2$, we note that the irreducible representations of a product of compact Lie groups are always given as outer tensor products of irreducible representations of the factors.  We also recall that an outer tensor product of two irreducible representations is orthogonal if the two factors are either both orthogonal or both symplectic, and the outer tensor product is symplectic if and only if one of the representations is symplectic and the other is orthogonal.

As mentioned in the previous subsection, we need to classify the conjugacy classes of images of homomorphisms from $U$ to $G\times G$.  To relate this problem to representation theory, we use the following theorem of Mal'cev.

\begin{theorem}\label{Mal'cev}Let $f_1, f_2:U\rightarrow G$ with $$G\in\{ SU(n), Sp(n), SO(2n+1)\}$$ be interpreted as complex, symplectic, or orthogonal representations.  If the representations are equivalent, then the images are conjugate in $G$.

\end{theorem}

The irreducible representations for every compact simply connected simple Lie group have been completely classified.  For $SU(2)$, we have the following proposition.

\begin{proposition}\label{sp1irrep} For each $n \geq 1$, $SU(2)$ has a unique irreducible representation of dimension $n$.  When $n$ is even this representation is symplectic, and when $n$ is odd this representation is orthogonal.

\end{proposition}

We will use the standard notation $\phi_i:SU(2)\rightarrow SU(i+1)$ to denote the unique $i+1$-dimensional irreducible representation of $SU(2)$ and $\phi_{ij}$ to denote $\phi_i\otimes \phi_j:U\rightarrow SU( (i+1)(j+1)).$  We note that, if $S^1 = \{\diag(z, \overline{z}): z\in \mathbb{C}\text{ and }|z|=1\}$ denotes the standard maximal torus of $SU(2)$, then $\phi_i(S^1) = \{\diag(z^i, z^{i-2},..., z^{-i})\}\subseteq SU(i+1)$.

Some of the lower dimensional $\phi_i$ are more commonly known.  The representation $\phi_1:SU(2)\rightarrow SU(2)$ is the inclusion map, $\phi_2:SU(2)\rightarrow SO(3)\subseteq SU(3)$ is the canonical double cover map, and $\phi_4:SU(2)\rightarrow SO(5)\subseteq SU(5)$ is the Berger embedding of $SO(3)$ into $SO(5)$.  The representation $\phi_6$ has image $B_{max}\subseteq SO(7)$, mentioned just after Theorem \ref{main}.  The representation $\phi_{11}:SU(2)^2\rightarrow SO(4)\subseteq SU(4)$ is the canonical double cover map.

Proposition \ref{sp1irrep} implies that all irreducible representations of both $SU(2)$ and $U$ are either orthogonal or symplectic, so $\phi_i = \overline{\phi}_i$.  Then Proposition \ref{symp} implies that a representation of either $SU(2)$ or $U$ is orthogonal (symplectic) iff every irreducible symplectic (orthogonal) subrepresentation appears with even multiplicity.

\subsection{The Octonions and Clifford Algebras}\label{oct}

The octonions $\mathbb{O}$, sometimes referred to as the Cayley numbers, are an $8$-dimensional non-associative normed division algebra over $\mathbb{R}$.  The octonions are alternative, meaning that the subalgebra generated by any two elements of $\mathbb{O}$ is associative.  In fact, if $x,y\in \mathbb{O}$, then the algebra generated by $x$ and $y$ is isomorphic to either $\mathbb{R}$, $\mathbb{C}$, or $\mathbb{H}$.  All of the necessary background may be found in \cite{GVZ}.  We will follow the conventions in \cite{Ke2}.

Viewing $\mathbb{O} = \mathbb{H} + \mathbb{H}l$, where $\mathbb{H}$ denotes the division algebra of quaternions, the multiplication is defined by $$(a+bl)\cdot (c+dl) = (ac-\overline{d}b) +(da + b\overline{c})l .$$

We use the canonical basis $$\{e_0 = 1, e_1 = i, e_2 = j, e_3 = k, e_4 = l, e_5 = il, e_6 = jl, e_7=kl\},$$ which we declare to be orthonormal.  Then the multiplication table is given in Table \ref{table:cayleymult}, where the entries are of the form $(\text{row}) \cdot(\text{column}).$

\begin{table}[h]

\begin{center}

\begin{tabular}{|c||c|c|c|c|c|c|c|}

\cline{2-8}
\multicolumn{1}{c|}{} & $e_1$ & $e_2$ & $e_3$ & $e_4$ & $e_5$ & $e_6$ & $e_7$\\ \hline

\hline

$e_1$ & $-1$ & $e_3$ & $-e_2$ & $e_5$ & $-e_4$ & $-e_7$ & $e_6$  \\ \hline

$e_2$ & $-e_3$ & $-1$ & $e_1$ & $e_6$ & $e_7$ & $-e_4$ & $-e_5$ \\ \hline

$e_3$ & $e_2$ & $-e_1$ & $-1$ & $e_7$ & $-e_6$ & $e_5$ & $-e_4$\\ \hline

$e_4$ & $-e_5$ & $-e_6$ & $-e_7$ & $-1$ & $e_1$ & $e_2$ & $e_3$\\ \hline

$e_5$ & $e_4$ & $-e_7$ & $e_6$ & $-e_1$ & $-1$ & $-e_3$ & $e_2$\\ \hline

$e_6$ & $e_7$ & $e_4$ & $-e_5$ & $-e_2$ & $e_3$ & $-1$ & $-e_1$\\ \hline

$e_7$ & $-e_6$ & $e_5$ & $e_4$ & $-e_3$ & $-e_2$ & $e_1$ & $-1$\\ \hline

\end{tabular}

\caption{Multiplication table for octonions}\label{table:cayleymult}

\end{center}

\end{table}

We use the octonions to construct an explicit embedding of $Spin(7)$ into $SO(8)$.  We will use this embedding to understand biquotients of the form $Spin(7)\bq SU(2)^2$.

\

Consider the Clifford algebra $$Cl_8 \cong \mathbb{R}^{16\times 16} =\left\{ \begin{bmatrix} A & B\\ C&D\end{bmatrix}: A,B,C,D\in \mathbb{R}^{8\times 8}\right\}.$$  We  linearly embed $\mathbb{O}$ into $Cl_8$ using left multiplication: if $L_x\in \mathbb{R}^{8\times 8}$ denotes left multiplication by $x\in \mathbb{O}$, then we identify $x\in \mathbb{O}$ with $\widehat{x} = \begin{bmatrix} 0 & -L_{\overline{x}}\\ L_x & 0\end{bmatrix}$.  A simple calculation shows $\widehat{x}\widehat{y} = \begin{bmatrix} -L_{\overline{x}} L_y & \\ & -L_x L_{\overline{y}} \end{bmatrix}$.  Polarizing the identity $L_{\overline{x}} L_x = \langle x, x\rangle I_8$, which follows from the fact that $\mathbb{O}$ is normed and alternative, we conclude $L_{\overline{x}} L_y + L_{\overline{y}} L_x = 2\langle x,y\rangle I_8$.  From here, a simple calculation shows \begin{equation}\label{cliffcommute} \widehat{x}\widehat{y} + \widehat{y}\widehat{x} = -2\langle x,y\rangle I_{16}. \end{equation}

We let $S^7 \subseteq \mathbb{O}$ denote the unit length octonions.  For $v\in S^7$, we note that $\widehat{v}\in O(16)$, which follows from the fact that $L_v:\mathbb{O}\rightarrow \mathbb{O}$ is an isometry.  Further, we also point out that $\widehat{v} \widehat{v} = \begin{bmatrix} -L_{\overline{v}} L_v & \\ & -L_v L_{\overline{v}}\end{bmatrix} = -I_{16}$.  It follows that $\widehat{v}^{-1} = -\widehat{v}$.

\begin{proposition}\label{reflection}  If $v\in S^7$, then conjugation by $\widehat{v}$ fixes $\widehat{v}$ and acts as $-1$ on $\widehat{v}^\bot\subseteq \widehat{\mathbb{O}}$.

\end{proposition}

\begin{proof}
For $x\in \mathbb{O}$, decompose it as $x = \lambda v + x^\bot$ with $\lambda\in\mathbb{R}$ and $x^\bot$ orthogonal to $v$.  Then, using \eqref{cliffcommute}, we have \begin{align*} \widehat{v}\widehat{x} \widehat{v}^{-1} &= -\widehat{v} \widehat{x} \widehat{v} \\ &= -\lambda \widehat{v}^3 - \widehat{v} \widehat{x^\bot} \widehat{v}\\ &= \lambda \widehat{v} - \widehat{v}\left(-2\langle x^\bot, v\rangle I_{16} - \widehat{v}{x^\bot}\right)\\ &= \lambda\widehat{v} + \widehat{v}^2 \widehat{x^\bot}\\ &= \lambda \widehat{v} - \widehat{x^\bot}.\end{align*}  The result follows.

\end{proof}

Consider $S^6\subseteq S^7\subseteq \mathbb{O}$ consisting of the unit length purely imaginary octonions.  Let $H$ denote the subgroup of $Cl_8$ generated by pairs of elements in $\widehat{S^6}$.  Since $\widehat{S^7}\subseteq O(16)$, it follows that $H\subseteq O(16)$.  Further, since $H$ is generated by $$H' = \left\{\widehat{v}\widehat{w} = \begin{bmatrix} L_v L_w & 0 \\ 0 & L_v L_w\end{bmatrix}: v,w\in S^6\right\},$$ we see that, in fact, $H$ is naturally a subgroup of $\Delta O(8)\subseteq O(8)\times O(8)\subseteq O(16)$.

In fact, $H$ is connected, so is a subgroup of $\Delta SO(8)\subseteq \Delta O(8)$.  To see this, recall that the group generated by a path connected subset containing the identity is path connected, and then simply note that $I = \widehat{i}\, \widehat{-i}\in H'$.

Now, consider the map $\pi:H\rightarrow SO(\widehat{Im\mathbb{O}}) = SO(7)$ given by $\pi(h)\widehat{y} = h \widehat{y} h^{-1}$.   

\begin{proposition}\label{Hidentity} The map $\pi$ is a double cover, so $H$ is isomorphic to $Spin(7)$.

\end{proposition}

\begin{proof}

We first show that $\pi$ has image contained in $SO(7)$.  For $h = \widehat{v}\widehat{w}\in H'$, we see that, by Proposition \ref{reflection}, conjugation by $h$ corresponds to a reflection along the $w$ axis followed by a reflection along the $v$ axis.  This fixes $\operatorname{span}\{\widehat{v},\widehat{w}\}^\bot\subseteq \widehat{\mathbb{O}}$ and rotates the plane spanned by $\widehat{v}$ and $\widehat{w}$ by twice the angle between $\widehat{v}$ and $\widehat{w}$.  Since $\widehat{1}\in \operatorname{span}\{\widehat{v},\widehat{w}\}^\bot$ for any $\widehat{v},\widehat{w} \in \widehat{S^6}$, $\pi(h)$ really is an element of $SO(7)$, so $\pi(H')\subseteq SO(7)$.  It follows that $\pi(H)\subseteq SO(7)$.

Also, $\pi$ is surjective.  This follows because $SO(7)$ is generated by pairs of reflections.

Finally, we note that $\ker \pi$ consists of precisely two elements.  To see this, note that $\widehat{i}\,\widehat{i} = -I_{16}\in H'$ and $\widehat{i}\,\widehat{-i} = I_{16}\in H'$, but $\pi(-I_{16}) = \pi(I_{16})= I$.  Thus, $\ker \pi$ contains at least $2$ elements.  On the other hand, because $H$ is connected and $\pi_1(SO(7))\cong \mathbb{Z}/2\mathbb{Z}$, $\ker \pi$ contains at most $2$ elements.

\end{proof}

With this description of $H = Spin(7)\subseteq SO(8)$, we now work towards identifying the maximal torus $T_H = T^3\subseteq Spin(7)\subseteq SO(8)$ and the projection $\pi:T_H\rightarrow T^3\subseteq SO(7)$.

\begin{proposition}\label{Hinso8} Suppose $v,w\in S^6\subseteq Im\mathbb{O}$ are independent with angle $\alpha$ between them and let $\mathbb{H}_{vw}$ denote the subalgebra of $\mathbb{O}$ generated by $v$ and $w$.  Let $0\neq x\in \mathbb{H}_{vw}^\bot\subseteq \mathbb{O}$.  Then $L_v L_w$ preserves each of the $2$-planes 

\begin{center}\begin{tabular}{lcl} $P_1 = \operatorname{span}\{1, vw \}$ & & $P_2 = P_1^\bot\subseteq \mathbb{H}_{uv}$ \\ $P_3 = L_x P_1$ & & $P_4 = L_x P_2$\end{tabular} \end{center} and rotates each $P_i$ through an angle $\alpha$.

\end{proposition}

\begin{proof}

Since $v$ and $w$ are purely imaginary and independent, $\mathbb{H}_{vw}\cong \mathbb{H}$.  In particular, on $P_1, P_2\subseteq \mathbb{H}_{vw}$, $L_v L_w = L_{vw}$.  Because the automorphism group of $\mathbb{H}$ acts transitively on the set of pairs of orthogonal unit length purely imaginary quaternions, we may assume $v = i$, $w = \cos(\alpha) i + \sin(\alpha) j$, and so, $vw = -\cos(\alpha) +  \sin(\alpha) k$.  Now a simple calculation shows that $L_{vw}$ rotates $P_1 =\operatorname{span}\{ 1, k\}$ and $P_2 = \operatorname{span}\{ i,j\}$ by an angle of $\alpha$.

For $P_3$ and $P_4$, since $x$ and $v$ are orthogonal, we see from \eqref{cliffcommute} that $\widehat{v}\widehat{x} = - \widehat{x}\widehat{v}$.  The top left $8\times 8$ block of the equation $\widehat{v}\widehat{x} = -\widehat{x}\widehat{v}$ is $-L_{\overline{v}}L_x =  L_{\overline{x}} L_v$.  Since $v$ and $x$ are purely imaginary, this shows $L_v$ and $L_x$ anti-commute.  Similarly, $L_w$ and $L_x$ anti-commute.

It follows that $L_vL_w P_3 = L_v L_w L_x P_1 = L_x L_v L_w P_1$ is a rotation by angle $\alpha$ as well, and similarly for $P_4 = L_x P_2$.

\end{proof}

For $n=1,2,3$, we set $v = e_{2n-1}$ and $w= -(cos(\theta)e_{2n-1} +\sin(\theta) e_{2n})$, so $\widehat{v}\widehat{w}\in H$.  From the proof of Proposition \ref{Hidentity}, we know that $\pi(\widehat{v}\widehat{w}) = \pi(-\widehat{v} \cdot \widehat{w})$ rotates $e_{2n-1}$ towards $e_{2n}$ by an angle $2\theta$.  On the other hand, using Proposition \ref{Hinso8}, we may compute the matrix form of $L_v L_w \in H\subseteq SO(8)$.

We work this out in detail when $n = 3$, so $$v = e_5 = il \text{ and } w = -(\cos(\gamma)(il) + \sin(\gamma)(jl)).$$  Then, $P_1 = \operatorname{span}\{1, k\}$.  Then $L_v L_w 1 = \cos(\theta) + \sin(\gamma) k$, so $L_v L_w$ rotates $1$ towards $k$.

On $P_2 = \operatorname{span}\{il, jl\}$, we compute \begin{align*} L_v L_w (il) &= -(il)(\cos(\gamma)(il)^2 + \sin(\gamma)(jl)(il)) \\ &= -(il)(-\cos(\gamma) +\sin(\gamma)k) = \cos(\gamma)(il) - \sin(\gamma)(il)k\\ &= \cos(\gamma)(il) -\sin(\gamma)(jl),\end{align*} so $L_vL_w$ rotates $il$ towards $-jl$.

On $P_3 = i\operatorname{span}\{1,k\} = \operatorname{span}\{i, j\}$, we compute \begin{align*} L_v L_w i &= -(il)(\cos(\gamma)(il)i + \sin(\gamma)(jl)i)\\ &= -(il)(\cos(\gamma)l + \sin(\gamma)kl)\\ &= -\cos(\gamma) (il)l -\sin(\gamma)(il)(kl)\\ &= \cos(\gamma)i -\sin(\gamma)j.\end{align*}  Thus, $L_vL_w$ rotates $i$ towards $-j$.

Finally, on $P_4 = i\operatorname{span}\{il , jl\} = \operatorname{span}\{l, kl \}$< we have \begin{align*}  L_v L_w(l) &= -(il)(\cos(\gamma)(il)l + \sin(\gamma)(jl)l)\\ &= -(il)(-\cos(\gamma)i -\sin(\gamma)j)\\ &= \cos(\gamma)(il)i + \sin(\gamma)(il)j\\ &= \cos(\gamma) l -\sin(\gamma)kl.\end{align*}  Thus, $L_v L_w$ rotates $l$ towards $-kl$.  

Putting this all together, it follows that when $$v = il \text{ and } w = -(\cos(\gamma) (il) + \sin(\gamma)(jl)),$$ that $L_v L_w$ has the matrix form $$ A_3 = \begin{bmatrix} \cos\gamma & & & -\sin\gamma & & & & \\ & \cos \gamma & \sin\gamma &  & & & &\\ & -\sin\gamma & \cos\gamma & & & & & \\ \sin\gamma & & &\cos\gamma & & & & \\ & & & & \cos\gamma & & & \sin\gamma \\ & & & & & \cos\gamma & \sin\gamma & \\ & & & & & -\sin\gamma & \cos\gamma & \\ & & & & -\sin\gamma & & & \cos\gamma \end{bmatrix}.$$ 

We let $R(\theta)$ denote the standard rotation matrix, $R(\theta) = \begin{bmatrix} \cos\theta & -\sin\theta \\ \sin\theta & \cos\theta\end{bmatrix}$, and we use the shorthand $R(\theta_1,..., \theta_k)$ to denote the block diagonal matrix $\diag(R(\theta_1), R(\theta_2), ..., R(\theta_k) )$  or $(R(\theta_1), ..., R(\theta_k),1)$ as appropriate.  In this notation, we have now proven that $\pi(A_3) = R(0,0,2\gamma)\in SO(7)$.

In an analogous fashion, one can show that when $v = i$ and $w = -(\cos(\alpha) i + \sin(\alpha) j)$, then $L_v L_w$ has the matrix form $$A_1 = \begin{bmatrix} \cos\alpha & & & \sin\alpha & & & & \\ & \cos \alpha & \sin\alpha &  & & & &\\ & -\sin\alpha & \cos\alpha & & & & & \\ -\sin\alpha & & &\cos\alpha & & & & \\ & & & & \cos\alpha & & & -\sin\alpha \\ & & & & & \cos\alpha & \sin\alpha & \\ & & & & & -\sin\alpha & \cos\alpha & \\ & & & & \sin\alpha & & & \cos\alpha \end{bmatrix}$$  with $\pi(A_1) = R(2\alpha, 0,0)$.  In addition, when $v = k$ and $w = -(\cos(\beta) k + \sin(\beta)l)$, then $L_v L_w$ has the matrix form $$A_2 = \begin{bmatrix} \cos\beta & & & & & & & \sin\beta\\ & \cos\beta & & & & & \sin\beta & \\  & & \cos\beta & & & -\sin\beta & & \\ & & & \cos\beta & \sin\beta & & & \\ & & & -\sin\beta & \cos\beta & & & \\ & & \sin\beta & & & \cos\beta & & \\ & -\sin\beta & & & & & \cos \beta & \\ -\sin\beta & & & & & & & \cos\beta \end{bmatrix}$$ with $\pi(A_2) = R(0,2\beta,0)$.

One can easily verify that the $A_i$ matrices commute, so they can be simultaneously conjugated to the standard maximal torus of $SO(8)$.  In fact, if we set $B = \begin{bmatrix} 0&0&0&1&0&0&0&1\\ 1&0&0&0&-1&0&0&0\\ 0&0&0&1&0&0&0&-1 \\ 1&0&0&0&1&0&0&0\\ 0&1&0&0&0&-1&0
&0\\  0&0&-1&0&0&0&1&0\\ 0&1&0&0&0&1&0&0\\ 0&0&-1&0&0&0&-1&0 \end{bmatrix},$ it is easy to verify that $BA_1 A_2 A_3 B^{-1} = R(\theta_1,\theta_2,\theta_3,\theta_4)$ with \begin{align*} \theta_1 = \alpha+\beta - \gamma &  & \theta_2 = \alpha-\beta -\gamma & \\ \theta_3 = \alpha-\beta + \gamma &  & \theta_4 = \alpha+\beta+\gamma &.\end{align*}  Thus, we have proven the following.

\begin{proposition}  Up to conjugacy, the maximal torus of $H$, $T_H\subseteq H= Spin(7)\subseteq SO(8)$ is given as $R(\theta_1,\theta_2,\theta_3,\theta_4)$ with $\theta_i$ as above.  In addition, the projection $\pi:H\rightarrow SO(7)$ maps $R(\theta_1,\theta_2,\theta_3,\theta_4)$ to $R(2\alpha, 2\beta, 2\gamma)$.

\end{proposition}

\

Because checking whether a biquotient is effectively free reduces to checking conjugacy of elements in a maximal torus, we must determine when two elements of $T_{Spin(7)}\subseteq Spin(7)\subseteq SO(8)$ are conjugate in $Spin(7)$.  To that end, recall the maximal torus of $SO(8)$ consists of elements of the form $R(\lambda_1, \lambda_2, \lambda_3, \lambda_4)$ and the Weyl group $W_{SO(8)}$ acts by arbitrary permutations of the $\lambda_i$ together with an even number of sign changes of the $\lambda_i$.

We also note that, with $\theta_i$ defined as above, that $\theta_1 + \theta_3 = \theta_2 + \theta_4$.  The Weyl group of $Spin(7)$, $W_{Spin(7)}$, acts as arbitrary permutations of $(\alpha, \beta, \gamma)$ as well as an arbitrary number of sign changes.

\begin{proposition}\label{spinconjugate}  Two elements of $T_H\subseteq H\subseteq SO(8)$, defined by the equation $\theta_1+\theta_3 = \theta_2+\theta_4$, are conjugate in $Spin(7)$ iff there is an element in $W_{SO(8)}$ which preserves $T$ and maps the first element to the second.

\end{proposition}

\begin{proof}

We first note that the action of every element of $W_{Spin(7)}$ on $T$ is the restriction of the action of some element in $W_{SO(8)}$ which preserves $T$.  To see this, note it is enough to check it on a generating set of $W_{Spin(7)}$.  The element which interchanges $\alpha$ and $\beta$ is the restriction of the element of $W_{SO(8)}$ which interchanges $\theta_2$ and $\theta_3$ and negates them both, so preserves $T$.  Similarly, interchanging $\beta$ and $\gamma$ corresponds to swapping $\theta_1$ and $\theta_3$.  Finally, the element of $W_{Spin(7)}$ which negates $\beta$ corresponds to simultaneously interchanging $\theta_1$ and $\theta_2$ and interchanging $\theta_3$ and $\theta_4$.  It is easy to see that these elements generate all of $W_{Spin(7)}$, and so this proves the ``only if'' direction.

To prove the ``if'' direction, we first note that $|W_{Spin(7)}| = 48$, so at least $48$ elements of $W_{SO(8)}$ preserve $T$.  We now prove there are no more.

Consider the action of $W_{SO(8)}$ on the set of all rank $3$ sub-tori of the maximal torus of $SO(8)$.  Because $W_{SO(8)}$ acts as an arbitrary permutation of the $\theta_i$, the orbit of the maximal torus of $Spin(7)\subseteq SO(8)$ contains at least the $3$ tori defined by the equations $\theta_1 + \theta_3 = \theta_2 + \theta_4$, $\theta_1 + \theta_2 = \theta_3 + \theta_4$, and $\theta_1+\theta_4 = \theta_2+\theta_3$.  In addition, because $W_{SO(8)}$ also acts by changing an even number of signs, the torus defined by the equations $-\theta_1-\theta_3 = \theta_2 + \theta_4 $ is also in the orbit.  By the orbit-stabilizer theorem, we have $|\text{Orbit}||\text{Stabilizer}| = |W_{SO(8)}| = 192$.  Since we have already shown the stabilizer has a subgroup of size $48$ and that the order of the orbit is at least $4$, it follows that the order of the stabilizer is precisely $48$.

\end{proof}

\section{\texorpdfstring{Biquotients of the form $SO(7)\bq SU(2)^2$ and $Spin(7)\bq SU(2)^2$}{Biquotients of the form SO(7)/SU(2)SU(2) and Spin(7)/SU(2)SU(2)}}

In this section, we prove Theorem \ref{main} in case of $G = SO(7)$ and $G = Spin(7)$

We begin by listing all homomorphisms, up to equivalence, from $SU(2)$ and $SU(2)^2$ into $SO(7)$.  Because $SU(2)$ is simply connected, every such homomorphism lifts to $Spin(7)$.

To determine which give rise to effectively free biquotient actions, we first classify all effectively free biquotient actions of $SU(2)$ on $Spin(7)$, finding precisely $12$.  We then check directly that all $12$ of these descend to effectively free actions of $SU(2)$ on $SO(7)$.

We use the classification of biquotients of the form $Spin(7)\bq SU(2)$ and Lemma \ref{restest} to classify biquotients of the form $Spin(7)\bq SU(2)^2$.  It follows from the classification that, with one exception, for each effectively free biquotient action of $SU(2)^2$ on $Spin(7)$, either the point $(-I,I)$ or $(I,-I)\in H^2\subseteq SO(8)^2$ is in the image of $SU(2)^2$.  Proposition \ref{easyhalf} then implies that each of these actions descends to an effectively free action of $SU(2)^2$ on $SO(7)$.  The exceptional case is easily checked.

\

Classifying homomorphisms from $SU(2)$ into $SO(7)$ is simply classifying all orthogonal $7$-dimensional representations of $SU(2)$.  To begin with, we note there is a natural bijection between partitions of $7$ and $7$-dimensional representations of $SU(2)$.  As mentioned in Section 2, Proposition \ref{symp} implies that a sum of representations of $SU(2)$ is orthogonal if and only if each symplectic $\phi_i$, that is, those with $i$ odd, appears with even multiplicity.  Thus, we seek partitions in which every even number appears an even number of times.  Compiling these, we obtain Table \ref{table:so7homo}.

\begin{table}[ht]

\caption{Nontrivial homomorphisms from $SU(2)$ into $SO(7)$}\label{table:so7homo}

\begin{center}

\begin{tabular}{|c|c|}

\hline

Representation & Image of $T_{SU(2)}\subseteq SO(7)$\\

\hline

$4\phi_0 + \phi_2$ & $R(2\theta,0,0) $\\

$3\phi_0 + 2\phi_1$ & $R(\theta,\theta, 0)$\\

$2\phi_0 + \phi_4$ & $R(4\theta,2\theta,0) $\\

$\phi_0 + 2\phi_2$ & $R(2\theta,2\theta,0)$\\

$2\phi_1 + \phi_2$ & $R(2\theta, \theta, \theta) $\\

$\phi_6$ & $R(6\theta, 4\theta, 2\theta) $\\

\hline

\end{tabular}

\end{center}

\end{table}

We recorded the image of the maximal torus, which is actually a subset of $SO(6)$, as this will be essential for determining whether a given biquotient action is effectively free or not.

For each entry in Table \ref{table:so7homo}, after lifting to $Spin(7)$ and including this into $SO(8)$, we obtain an $8$-dimensional orthogonal representation of $SU(2)$.

For example, one can easily verify that if one chooses $\alpha = \theta, \beta = \gamma = 0$, then image of the maximal torus of $SU(2)$ is, in this case, $R( \alpha, \alpha,\alpha,\alpha)\in Spin(7)\subseteq SO(8)$ which projects to $\diag(R(2\alpha), 1,1,1,1,1)\in SO(7)$.  Hence, the lift of $4\phi_0 + \phi_2$ is, up to conjugacy, $4\phi_1$.  Continuing in this fashion, we obtain Table \ref{table:su2spin}.

\begin{table}[ht]

\caption{Lifts of homomorphisms from $SU(2)$ to $SO(7)$ into $Spin(7)\subseteq SO(8)$}\label{table:su2spin}

\begin{center}

\begin{tabular}{|c|c||c|c|c|c|}

\hline

Name &Representation &  $\alpha$ & $\beta$ & $\gamma$ & Image of $T_{SU(2)}\subseteq H\subseteq SO(8)$ \\

\hline\hline

$A$ & $4\phi_0 + \phi_2$  &  $\theta$ & $0$ & $0$ & $R(\theta,\theta,\theta,\theta) $ \\

$B$ & $3\phi_0 + 2\phi_1$ &  $\frac{1}{2}\theta$ & $\frac{1}{2}\theta$ & $0$ & $R(\theta, 0,0,\theta)$ \\

$C$ & $2\phi_0 + \phi_4$ & $2\theta$ & $\theta$ & $0$ & $R(3\theta, \theta, \theta, 3\theta)$ \\

$D$ & $\phi_0 + 2\phi_2$ & $\theta$ & $\theta$ & $0$ & $R(2\theta, 0,0,2\theta) $ \\

$E$ & $2\phi_1 + \phi_2$ & $\frac{1}{2}\theta$ & $\frac{1}{2}\theta$ & $\theta$ & $R(0,-\theta,\theta,2\theta) $\\

$F$ & $\phi_6$ & $\theta$ & $2\theta$ & $3\theta$ & $ R(0,-4\theta, 2\theta, 6\theta) $\\ 

\hline

\end{tabular}

\end{center}

\end{table}

In a similar fashion, one can classify all homomorphisms, up to equivalence, from $U = SU(2)^2$ to $SO(7)$.   Recalling that the dimension of $\phi_{ij}$ is $(i+1)(j+1)$, one first tabulates a list of all representations of $U$ of total dimension $7$.  Since orthogonal representations correspond to those representations for which every $\phi_{ij}$ with $i$ and $j$ of different parities appears with even multiplicity, one easily finds all $7$-dimensional orthogonal representations of $SU(2)^2$.  Table \ref{table:so7partition} records all of the representations with finite kernel up to interchanging the two $SU(2)$ factors; those with infinite kernel are a composition $SU(2)^2\rightarrow SU(2)\rightarrow SO(7)$ where the first map is one of the two projection maps and the second is listed in Table \ref{table:su2spin}.

\begin{table}[ht]

\caption{Immersions from $SU(2)^2$ into $SO(7)$ and $Spin(7)\subseteq SO(8)$}\label{table:so7partition}

\begin{center}

\begin{tabular}{|c|c|c|}

\hline

Representation & Image of  $T_U\subseteq SO(7)$ & Lift $T_U\subseteq H\subseteq SO(8)$  \\

\hline

$3\phi_{00} + \phi_{11}$ & $R(\theta + \phi, \theta -\phi, 0) $ & $ R(\theta, \phi,\phi,\theta)  $\\

$\phi_{00} + \phi_{20} + \phi_{02}$ & $R(2\theta, 2\phi, 0) $ & $ R(\theta + \phi, \theta - \phi, \theta - \phi, \theta + \phi) $\\

$2\phi_{10} + \phi_{02}$ & $R(2\phi, \theta, \theta) $ & $ R(\phi, \phi-\theta, \phi, \theta + \phi)  $\\

$\phi_{11} + \phi_{02}$ & $R(2\phi , \theta + \phi, \theta - \phi) $ & $ R(\theta + \phi, 0, \phi - \theta, 2\phi) $\\

\hline

\end{tabular}

\end{center}

\end{table}

We now investigate which pairs of homomorphisms give rise to effectively free actions.  We begin with the homomorphisms with domain $SU(2)$, found in Table \ref{table:so7homo}.

\begin{proposition}\label{so7class}

Suppose $f=(f_1,f_2):SU(2)\rightarrow Spin(7)^2\subseteq SO(8)^2$ with both $f_1$ and $f_2$ nontrivial.  Then $f$ induces an effectively free biquotient action of $SU(2)$ on $Spin(7)$ if and only if up to order, we have $$(f_1,f_2) \in \{(A,B), (A,D), (A,E), (A,F), (C,E), (D,E) \}$$
\end{proposition}

\begin{proof}

Recall that a biquotient action defined by $(f_1, f_2)$ is effectively free if and only if for all $z \in T_{SU(2)}$, if $f_1(z)$ is conjugate to $f_2(z)$ in $G$, then $f_1(z) = f_2(z) \in Z(G)$. It immediately follows that $f_1$ and $f_2$ must be distinct, and that if either $f_1$ or $f_2$ is the trivial homomorphism, then the action is automatically free, which accounts for $6$ homogeneous examples.  This leaves $\binom{6}{2} = 15$ pairs of homomorphisms to check.  We present a few of the relevant calculations.

As shown in Proposition \ref{spinconjugate}, two elements in $T\subseteq Spin(7)$ are conjugate in $Spin(7)$ iff they are conjugate in $SO(8)$ by an element which preserves $T$.  We recall that the Weyl group acts on an element $R(\theta_1,...,\theta_4)$ in the maximal torus of $SO(8)$ by arbitrary permutations and an even number of sign changes.

For the pair $(D, E)$, the two images of the maximal tori are $R(2\theta, 0,0,2\theta)$ and $R(0,-\theta, \theta, 2\theta)$.  If, up to permutations and an even number of sign changes, these are equal, then we must either have $\theta = 0$ or $2\theta = 0$.  Of course, the first case only arises when both matrices are the identity matrix, so we focus on the second case.  If $2\theta = 0$, the first matrix becomes $I$, which then forces the second to be $I$ as well.  Hence, these two elements are conjugate iff they are both the identity, and so the action is effectively free.

On the other hand, the pair $(C,D)$ does not give rise to an effectively free action.  The images of the two maximal tori are $R(3\theta, \theta, \theta, 3\theta)$ and $R(2\theta,0,0,2\theta)$.  Choosing $\theta = 2\pi /3$, we obtain the matrices $$R(0, 2\pi/3, 2\pi/3,0) \text{ and } R(4\pi/3, 0,0,4\pi/3) = R(-2\pi/3, 0,0,-2\pi/3).$$  Consider the element $w\in W_{SO(8)}$ with $w(\theta_1,\theta_2,\theta_3, \theta_4) = -(\theta_2, \theta_1, \theta_4, \theta_3)$.  The one easily sees that $w$ maps the first torus element to the second and $w$ preserves the maximal torus of $Spin(7)$.  By Proposition \ref{spinconjugate}, these two torus elements are conjugate in $Spin(7)$.  Since they are not in $Z(Spin(7)) = \{\pm I\}\subseteq SO(8)$, it follows that this action is not effectively free.

\end{proof}

One can easily verify that each pair listed in Proposition \ref{so7class} descends to an effectively free action on $SO(7)$.  For example, for the pair $(D,E)$, the image of the maximal torus in $SO(7)$ consists of matrices of the form $R(2\theta, 2\theta, 0)$ and $R(2\theta, \theta,\theta)$.  If these matrices are conjugate in $SO(7)$, then either $\theta = 0$, which forces both matrices to be the identity, or $2\theta = 0$.  Substituting this in to the first matrix gives the identity matrix, which then forces the second matrix to be the identity.  Hence, these matrices are conjugate iff they are both the identity, so the action is free.

\

We now classify all biquotients of the form $Spin(7) \bq SU(2)^2$.  Our main tool is Lemma \ref{restest} which asserts that for any $f=(f_1,f_2):SU(2)^2\rightarrow Spin(7)^2$, if $(f_1,f_2)$ defines an effectively free action, then after restricting $f$ to a maximal torus $T^2\subseteq SU(2)^2$ with parameters $\theta$ and $\phi$, that setting $\theta=1$, $\phi=1$, or $\theta=\phi$ must result in a pair of homomorphisms from Proposition \ref{so7class}.

\begin{proposition}  Suppose $(f_1,f_2):SU(2)^2\rightarrow Spin(7)^2$ defines an effectively free action of $SU(2)^2$ on $Spin(7)$ with $f_1$ nontrivial.  Then either $f_2$ is trivial, or, up to interchanging $f_1$ and $f_2$, the pair $(f_1,f_2)$ is a lift of a pair in  Table \ref{table:su2class}.

\end{proposition}

\begin{proof}

A homomorphism $f:U=SU(2)^2\rightarrow G^2$ is nothing but a pair of homomorphism $f_i:U\rightarrow G$.  We break the proof into cases depending on whether or not each $f_i$ has finite kernel or infinite kernel.  If $f_i$ has finite kernel, it is, up to interchanging the two $SU(2)$ factors, the lift of a representation found in Table \ref{table:so7partition}.  On the other hand, $f_i$ has infinite kernel (but is non-trivial), it is given as the lift of a projection to either factor composed with an entry in Table \ref{table:su2spin}.

To facilitate the use of Lemma \ref{restest}, we record in Table \ref{table:restriction}, for each representation in Table \ref{table:so7partition}, the restriction of this representation to the three natural $SU(2)$ subgroups.

\begin{table}[h]

\caption{Restrictions of $SU(2)^2\rightarrow Spin(7)$ to $SU(2)$ subgroups}\label{table:restriction}

\begin{center}

\begin{tabular}{|c||c|c|c|}

\hline

Representation & $\theta = 0$ & $\phi=0$ & $\theta = \phi$\\

\hline

$3\phi_{00} + \phi_{11}$ &  $B$ & $B$ & $A$\\

$\phi_{00} + \phi_{20} + \phi_{02}$ & $A$ & $A$ & $D$\\

$2\phi_{10} + \phi_{02}$ & $A$ & $B$ & $E$\\

$\phi_{11} + \phi_{02}$ & $E$ & $B$ & $D$ \\

\hline

\end{tabular}

\end{center}

\end{table}

We first handle the case where both $f_i$ have finite kernel.  The first two representations in Table \ref{table:restriction} are symmetric in $\theta$ and $\phi$, but the last two are not.  It follows that, in addition to checking all $\binom{4}{2} = 6$ pairs of representations in Table \ref{table:restriction}, that we also must check $(2\phi_{10} + \phi_{02}, \phi_{11} + \phi_{20})$.  With the exception of $(3\phi_{0} + \phi_{11}, \phi_{00} + \phi_{20} + \phi_{02})$, the remaining $6$ pairs have at least one non-effectively free restriction, that is, at least one restriction which is not found in Proposition \ref{so7class}.  In the exceptional case, one easily sees that $\theta = \frac{2\pi}{5} $ and $\phi = \frac{4\pi}{5}$ gives rise to a non-central conjugacy in $Spin(7)$, so the induced action is not effectively free.  This completes the case where both $f_i$ have finite kernel.

So, we assume $f_2$ has infinite kernel.  This implies that one of the restrictions is the trivial representation, while the other two must be the same.  This immediately rules out the case $f_1 = 3\phi_{00} + \phi_{11}$, because the only pair in Proposition \ref{so7class} containing $B$ is $(A,B)$, which then forces the $\theta =\phi$ restriction to be $(A,A)$, so the induced action is not effectively free.

Similarly, if $f_1 = \phi_{00} + \phi_{20} + \phi_{02}$, then $f_2 = E$ is the only possibility which is not ruled out by Lemma \ref{restest}.  In this case, one has $R(\theta+\phi, \theta - \phi, \theta-\phi, \theta+\phi)$ and $R(0,-\theta,\theta,2\theta)$.  If these are conjugate, then $\theta = \pm \phi \pmod{2\pi}$, either case of which gives rise to the pair $(D,E)$.  Hence, by Proposition \ref{so7class}, this action is effectively free.

If $f_1$ is one of the remaining two representations with finite kernel, then the $\phi=0$ restriction forces $f_2$ to either only depend on $\phi$, or $f_2 = A$ with the variable $\theta$.  In the latter case, one easily verifies that both choices of $f_1$ give rise to effectively free biquotients, we now assume $f_2$ depends only on $\phi$.  If $f_1 = 2\phi_{10} + \phi_{02}$, the Proposition \ref{so7class} implies $f_2 = D$, and, if $f_1 = \phi_{11} + \phi_{02}$, then $f_2 = A$.  One easily verifies that both of these give rise to effectively free actions.

\

We have now handled the case where $f_1$ has finite kernel, so we may assume $f_1$ and $f_2$ both have infinite kernel.  The $\phi =\theta$ restriction implies that the only cases which can possibly give rise to effectively free actions are those given in Proposition \ref{so7class}, where we assume $f_1$ only depends on $\theta$, while $f_2$ depends only on $\phi$.  One easily sees that the first five entries give rise to effectively free actions.

For example, for $(C,E)$, the two images of the maximal tori given by $R(3\theta, \theta,\theta,3\theta)$ and $R(0,-\phi , \phi ,2\phi )$.  They are not even conjugate in $SO(8)$, unless $\theta = 0$.

The last entry, $(D,E)$, does not give rise to an effectively free action.  The images of the maximal torus are $R(2\theta,0,0,2\theta)$ and $R(0,-\phi, \phi, 2\phi)$.  Setting $\theta = \pi/2$ and $\phi = \pi$ gives a non-central conjugacy in $Spin(7)$.

\end{proof}

We note that, as a corollary to the proof, for $9$ of the $10$ effectively free biquotients, the map $(f_1,f_2):SU(2)^2\rightarrow Spin(7)^2\subseteq SO(8)^2$ defining the action has  either $(-I,I)$ or $(I,-I)$ in its image, so Proposition \ref{easyhalf} implies each of these descends to a biquotient of the form $SO(7)\bq SU(2)^2$.  The exceptional case is given by the pair $(2\phi_{10} + \phi_{02}, \phi_{00} + 2\phi_{02})$.  As a representation into $SO(7)$, the two images of the maximal tori are $R(\theta, \theta, 2\phi)$ and $R(2\phi, 2\phi, 0)$.  If these are conjugate in $SO(7)$, then either $\theta = 0$ or $2\phi = 0$.  Either case forces the other case to hold, so these matrices are only conjugate when both are the identity.  It follows that this action of $SU(2)^2$ on $SO(7)$ is effectively free, finishing the proof of Theorem \ref{main} in the case of $G = SO(7)$ or $G = Spin(7)$.

\section{\texorpdfstring{$G = SU(4)$ or $G = \mathbf{G}_2\times SU(2)$}{G = SU(4) or G = G2 x SU(2)}}

In this section, we complete the proof of Theorem \ref{main} in the case of $G = SU(4)$ and $G = \mathbf{G}_2\times SU(2)$.

We begin with the case $G = SU(4)$.  We note that there are, up to equivalence, precisely 5 homomorphisms from $SU(2)$ to $SU(4)$, corresponding to the $5$ partitions of $4$.  These are listen in Table \ref{table:su4homo}.  Similarly, up to equivalence, there are precisely $2$ homomorphisms $SU(2)^2\rightarrow G$ with finite kernel: $\phi_{10} + \phi_{01}$ and $\phi_{11}$.  The image of $T_{SU(2)^2}$ is, respectively, $\diag(z,\overline{z},w,\overline{w})$ and $\diag(zw,\overline{z}\overline{w}, z\overline{w}, \overline{z} w)$.

\begin{table}[ht]

\caption{Homomorphisms from $SU(2)$ into $SU(4)$}\label{table:su4homo}

\begin{center}

\begin{tabular}{|c|c|}

\hline

Representation & Image of $T_{SU(2)}$\\

\hline

$4\phi_0$ & \diag(1,1,1,1)\\

$2\phi_0 + \phi_1$ & $\diag(z,\overline{z},1,1)$\\

$\phi_0 + \phi_2$ & $\diag(z^2,1,\overline{z}^2,1)$\\

$\phi_3$ & $\diag(z,z^3,\overline{z},\overline{z}^3)$\\

$2\phi_1$ & $\diag(z,\overline{z},z,\overline{z})$\\

\hline

\end{tabular}

\end{center}

\end{table}

As in the proof of the case $G = SO(7)$, we first classify all effectively free actions of $SU(2)$ on $SU(4)$.  The proof is carried out just as in Proposition \ref{so7class}, which the only change being that two diagonal matrices in $SU(4)$ are conjugate iff the entries are the same, up to order.

\begin{proposition}\label{su4class}

Suppose $f=(f_1,f_2):Sp(1)\rightarrow G^2$ with $f_1$ nontrivial.  Then $f$ induces an effectively free biquotient action of $SU(2)$ on $G$ if and only if either $f_2$ is trivial or, up to interchanging $f_1$ and $f_2$, $(f_1,f_2)$ is equivalent to $$(2\phi_0 + \phi_1, 2\phi_1) \text{ or }(\phi_0 + \phi_2, 2\phi_1).$$

\end{proposition}

We may now use Lemma \ref{restest} to complete the proof of Theorem \ref{main} in the case of $G = SU(4)$.

\begin{proposition}Table \ref{table:su2class} lists all of the inhomogeneous biquotients of the form $SU(4)\bq SU(2)^2$.

\end{proposition}

\begin{proof}

As in the proof of Theorem \ref{main} in the case of $G = SO(7)$, we only consider $(f_i, f_j)$ with $f_i\neq f_j$, and neither $f_i$ the trivial map.  Taking into account the symmetry of interchanging $z$ and $w$, we reduce the number to $15$ pairs to check.  Finally, Lemma \ref{restest} reduces this number down to $3$, of which only one does not give rise to an effectively free action.  We now provide the computations for these three cases.

Consider the action induced by $(\phi_{10} + \phi_{01}, \phi_{11})$, with maximal torus given by $$(\diag(z, \overline{z}, w, \overline{w}), \diag(zw, \overline{zw}, z\overline{w}, \overline{z}w)).$$  We set $w = z^3$ and choose $z$ to be a nontrivial $5$th root of unity.  Then the image on the left is $(z,z^4, z^3, z^2)$ while the image on the right is $(z^4, z, z^3, z^2)$.  These two matrices are clearly conjugate in $SU(4)$, but neither is an element of $Z(SU(4)) = \{\pm I\}$.

On the other hand, the action given by $(\diag(z, \overline{z}, 1, 1), \diag(w, \overline{w}, w, \overline{w}))$ is effectively free because if two such matrices are conjugate, we must have $w=1$, which then forces $z=1$.

Finally, the action given by $(\diag(z, \overline{z}, z, \overline{z}), \diag(w^2, 1, \overline{w}^2, 1))$ is effectively free since to be conjugate we require $z=1$, which then forces $w=1$.

\end{proof}

\

We now complete the proof of Theorem \ref{main} in the last case case, when $G = \mathbf{G}_2\times SU(2)$.

As mentioned in Section \ref{sec:background}, up to conjugacy, the identity map $SU(2)\rightarrow SU(2)$ is the unique non-trivial homomorphism.  It follows that, up to equivalence, there are precisely two non-trivial biquotient actions of $SU(2)$ on itself: left multiplication, and conjugation.  In particular, for any $U$ action on $SU(2)$, at least one $SU(2)$ factor of $U$ acts trivially.

The left multiplication action is clearly transitive, and thus, gives rise to a non-reduced biquotient.  If $SU(2)$ acts either trivially or by conjugation on itself, then $I$ is a fixed point of the action.  It follows that if neither factor of $U = SU(2)^2$ acts transitively on the $SU(2)$ factor of $G$, that, in fact, $U$ must act freely on $\mathbf{G}_2$.

As shown in \cite{KZ}, there is a unique effectively free biquotient action of $U$ on $\mathbf{G}_2$, giving rise to the exceptional symmetric space $\mathbf{G}_2/SO(4)$.  It follows that there are precisely three non-reduced biquotients of the form $\mathbf{G}_2\times SU(2)\bq U$:  either $U$ acts trivially on the $SU(2)$ factor of $G$ (giving the homogeneous space $(\mathbf{G}_2/SO(4))\times SU(2)$), or one of the two $SU(2)$ factors of $U$ acts by conjugation on the $SU(2)$ factor of $G$, giving the last two entries.  We note that, as mentioned in \cite{KZ}, the two normal $SU(2)$s in $SO(4)$ have different Dynkin indices in $\mathbf{G_2}$.  It follows that these two actions are not equivalent.





\bibliographystyle{plain}

\begin{thebibliography}{10}

\bibitem{AW}
S.~Aloff and N.~Wallach.
\newblock An infinite family of $7$-manifolds admitting positively curved
  {R}iemannian structures.
\newblock {\em Bull. Amer. Math. Soc.}, 81:93--97, 1975.

\bibitem{Baz1}
Y.~Bazaikin.
\newblock On a certain family of closed $13$-dimensional {R}iemannian manifolds
  of positive curvature.
\newblock {\em Sib. Math. J.}, 37(6):1219--1237, 1996.

\bibitem{Ber}
M.~Berger.
\newblock Les vari\'{e}t\'{e}s {R}iemanniennes homog\'{e}nes normales
  simplement connexes \'{a} courbure strictement positive.
\newblock {\em Ann. Scuola Norm. Sup. Pisa}, 15:179--246, 1961.

\bibitem{De}
O.~Dearricott.
\newblock A 7-manifold with positive curvature.
\newblock {\em Duke Math. J.}, 158:307--346, 2011.

\bibitem{D1}
J.~DeVito.
\newblock The classification of simply connected biquotients of dimension 7 or
  less and 3 new examples of almost positively curved manifolds.
\newblock {\em Thesis, University of Pennsylvania}, 2011.

\bibitem{DeV1}
J.~DeVito.
\newblock The classification of compact simply connected biquotients in
  dimensions 4 and 5.
\newblock {\em Diff. Geo. and Appl.}, 34:128--138, 2014.

\bibitem{DDRW}
J.~DeVito, R.~DeYeso III, M.~Ruddy, and P.~Wesner.
\newblock The classification and curvature of biquotients of the form
  ${S}p(3)/\!\!/{S}p(1)^2$.
\newblock {\em Ann. Glob. Anal. and Geo.}, 46(4):389--407, 2014.

\bibitem{Es1}
J.~Eschenburg.
\newblock New examples of manifolds with strictly positive curvature.
\newblock {\em Invent. Math.}, 66:469--480, 1982.

\bibitem{Es2}
J.~Eschenburg.
\newblock Freie isometrische aktionen auf kompakten {L}iegruppen mit positiv
  gekr$\ddot{\text{u}}$mmten {O}rbitr$\ddot{\text{a}}$umen.
\newblock {\em Schriftenr. Math. Inst. Universit$\ddot{\text{a}}$t
  M$\ddot{\text{u}}$nster}, 32, 1984.

\bibitem{EK}
J.~Eschenburg and M.~Kerin.
\newblock Almost positive curvature on the {G}romoll-{M}eyer 7-sphere.
\newblock {\em Proc. Amer. Math. Soc.}, 136:3263--3270, 2008.

\bibitem{FH}
W.~Fulton and J.~Harris.
\newblock {\em Representation Theory A First Course}.
\newblock Springer, 2004.

\bibitem{GrMe1}
D.~Gromoll and W.~Meyer.
\newblock An exotic sphere with nonnegative sectional curvature.
\newblock {\em Ann. Math.}, 100:401--406, 1974.

\bibitem{GVZ}
K.~Grove, L.~Verdiani, and W.~Ziller.
\newblock An exotic ${T^1 S^4}$ with positive curvature.
\newblock {\em Geom. Funct. Anal.}, 21:499--524, 2011.

\bibitem{KZ}
V.~Kapovitch and W.Ziller.
\newblock Biquotients with singly generated rational cohomology.
\newblock {\em Geom. Dedicata}, 104:149--160, 2004.

\bibitem{Ke2}
M.~Kerin.
\newblock Some new examples with almost positive curvature.
\newblock {\em Geometry and Topology}, 15:217--260, 2011.

\bibitem{Ke1}
M.~Kerin.
\newblock On the curvature of biquotients.
\newblock {\em Math. Ann.}, 352:155--178, 2012.

\bibitem{KT}
M.~Kerr and K.~Tapp.
\newblock A note on quasi-positive curvature conditions.
\newblock {\em Diff. Geo. and Appl.}, 34:63--79, 2014.

\bibitem{On1}
B.~O'Neill.
\newblock The fundamental equations of a submersion.
\newblock {\em Michigan Math. J.}, 13:459--469, 1966.

\bibitem{PW2}
P.~Petersen and F.~Wilhelm.
\newblock Examples of {R}iemannian manifolds with positive curvature almost
  everywhere.
\newblock {\em Geom. and Top.}, 3:331--367, 1999.

\bibitem{Ta1}
K.~Tapp.
\newblock Quasi-positive curvature on homogeneous bundles.
\newblock {\em J. Diff. Geo.}, 65:273--287, 2003.

\bibitem{Wa}
N.~Wallach.
\newblock Compact homogeneous riemannian manifolds with strictly positive
  curvature.
\newblock {\em Ann. of Math.}, 96:277--295, 1972.

\bibitem{W}
F.~Wilhelm.
\newblock An exotic sphere with positive curvature almost everywhere.
\newblock {\em J. Geom. Anal.}, 11:519--560, 2001.

\bibitem{Wi}
B.~Wilking.
\newblock Manifolds with positive sectional curvature almost everywhere.
\newblock {\em Invent. Math.}, 148:117--141, 2002.

\end{thebibliography}

\end{document}